\documentclass{article}
\usepackage{mystyle}
\begin{document}
\title{Reinhardt cardinals in inner models}
\author{Gabriel Goldberg}
\maketitle
\section{Introduction}
A cardinal \(\kappa\) is \textit{Reinhardt} if it is the critical point of an elementary embedding
from the universe of sets to itself. Kunen \cite{Kunen} famously refuted 
the existence of Reinhardt cardinals using the Axiom of Choice (AC).
It is a longstanding open problem whether Reinhardt cardinals
are consistent if AC is dropped.

Noah Schweber \cite{Schweber} introduced the notion of a \textit{uniformly supercompact cardinal},
a cardinal \(\kappa\) that is the critical point of an elementary embedding \(j : V\to M\)
such that \(M^\alpha\subseteq M\) for all ordinals \(\alpha\). He posed the
question of whether such a cardinal must be Reinhardt, and he also asked about
the consistency strength of uniformly supercompact cardinals. 
Both questions remain open, but this note makes some progress on the matter.

Say a cardinal $\kappa$ is \textit{weakly Reinhardt} if it is the critical point of an elementary embedding \(j : V\to M\)
such that \(j\restriction P(\alpha)\in M\) for all ordinals \(\alpha\). This condition is
equivalent to requiring that \(P(P(\alpha))\subseteq M\) for all ordinals \(\alpha\).
It seems to be weaker than demanding that \(M^{P(\alpha)}\subseteq M\) for all ordinals \(\alpha\).
\begin{repthm}{thm:inner_model}If there is a proper class of weakly Reinhardt cardinals, then there is an inner model with a proper class of Reinhardt cardinals.\end{repthm}
As a corollary, we obtain a consistency strength lower bound for 
a large cardinal that looks a bit more like Schweber's: for lack of a better term, say
\(\kappa\) is \textit{ultrafilter Reinhardt} if it is the critical point of an elementary
embedding \(j : V\to M\) such that for all ordinals \(\alpha\), \(M^\alpha\subseteq M\) 
and \(\beta(\alpha)\subseteq M\). Here \(\beta(X)\) denotes the set of ultrafilters on \(X\).
\begin{repprp}{thm:usc_reinhardt}If \(\kappa\) is ultrafilter Reinhardt,
    then \(\kappa\) is weakly Reinhardt.\end{repprp}
\subsection{Preliminaries}
Our background theory is
von Neumman-Bernays-G\"odel (NBG)
set theory without AC.
Even though we work without AC, for us a \textit{cardinal}
is an ordinal number that is not in bijection with any smaller ordinal. 
Of course, if AC fails, there are sets whose cardinality cannot be identified with a
cardinal in this sense. Still, for any set \(Y\), one can define
the \textit{Hartogs number of \(Y\)},
denoted by \(\aleph(Y)\), as the least cardinal \(\kappa\) such that
there is no injection from \(\kappa\) to \(Y\).

\section{The inner model \(N_\nu\)}
If $\nu$ is a cardinal and $X$ is a set, $\beta_\nu(X)$ denotes the set of $\nu$-complete ultrafilters on $X$.
In the context of choiceless large cardinal axioms, sufficiently complete ultrafilters on ordinals can
often be treated as ``idealized ordinals." The following lemma is a simple example of this phenomenon,
although the pattern runs quite a bit deeper than this.
\begin{lma}\label{lma:uf_wellorder}If there is a weakly Reinhardt cardinal, 
    then for all sufficiently large cardinals $\nu$, for any ordinal $\alpha$, $\beta_\nu(\alpha)$ can be wellordered.

\begin{proof}Assume not. Let $j : V\to M$ witness that $\kappa$ is weakly Reinhardt. 
    By transfinite recursion, define a sequence of ordinals $\delta_\xi$ for $\xi \in \text{Ord}$, 
    taking suprema at limit ordinals and, at successor stages, setting $\delta_{\xi+1}$ 
    equal to the least ordinal $\alpha > \aleph(\beta(\delta_\xi))$ such 
    that $\beta_{\delta_\xi}(\alpha)$ cannot be wellordered. 
    Let $\epsilon_\xi = (\delta_\xi)^M$. Then $j(\delta_\kappa) = \epsilon_{j(\kappa)} > \epsilon_{\kappa + 1}$. 
    For each $\gamma < \epsilon_{j(\kappa)+1}$, let $D_\gamma$ be the ultrafilter 
    on $\delta_{\kappa+1}$ derived from $j$ using $\gamma$, so 
    $D_\gamma = \{A\subseteq \delta_{\kappa+1} : \gamma\in j(A)\}$. 

    Note that the function \(\mathcal D(\gamma) = D_\gamma\) is simply definable
    from \(j\restriction P(\delta_{\kappa+1})\), and so \(d\in M\).
    For any $W\in \beta_{\epsilon_{j(\kappa)}}(\epsilon_{j(\kappa)+1})$, 
    $\mathcal D$ is constant on a set in $W$
    because \(W\) is $\epsilon_{j(\kappa)}$-complete
    and $\ran(\mathcal D)$ 
    has cardinality less than $\epsilon_{j(\kappa)}$.
    Indeed, \(|\ran(\mathcal D)| < \aleph^M(\beta(\delta_\kappa))\),
    the Hartogs number of \(\beta(\delta_\kappa)\) as computed in \(M\),
    since \(\ran(\mathcal D)\) is a wellorderable
    subset of \(\beta(\delta_\kappa)\) in \(M\).
    Moreover, \(\aleph^M(\beta(\delta_\kappa)) \leq \aleph^M(\beta(\epsilon_\kappa))\)
    since \(\epsilon_\kappa = \sup j[\delta_\kappa]\geq \delta_\kappa\),
    and \(\aleph^M(\beta(\epsilon_\kappa)) < \epsilon_{\kappa+1} < \epsilon_{j(\kappa)}\)
    by the definition of the ordinals \(\delta_\xi\) and the elementarity of \(j\). 
    
    Suppose $U\in \beta_{\delta_\kappa}(\delta_{\kappa+1})$,
    and we will show that for \(j(U)\)-almost all \(\gamma\), \(D_\gamma = U\). 
    Let \(D\) be the unique ultrafilter on \(\delta_{\kappa+1}\)
    such that \(D_\gamma = D\) for \(j(U)\)-almost all \(\gamma < \epsilon_{j(\kappa)+1}\).
    If $A\in U$, then for all 
    $\gamma\in j(A)$, $A\in D_\gamma$, and hence for $j(U)$-almost all $\gamma$, 
    $A\in D_\gamma$. 
    It follows that \(A\in D\). This proves \(U\subseteq D\), and so \(U = D\).
    Therefore $\mathcal D$ 
    is a surjection from the ordinal $\epsilon_{j(\kappa)+1}$ to $\beta_{\delta_\kappa}(\delta_{\kappa+1})$, 
    which contradicts that $\beta_{\delta_\kappa}(\delta_{\kappa+1})$ cannot be wellordered.
\end{proof}
\end{lma}
Let us now define the inner model in which weakly Reinhardt cardinals become Reinhardt.
Suppose \(\nu\) is a cardinal. 
Let \(\beta_\nu(\Ord) = \bigcup_{\alpha\in \Ord} \beta_\nu(\alpha)\) 
denote the class of \(\nu\)-complete ultrafilters on ordinals.
For any class \(C\), we denote the class of all subsets of \(C\) by \(P(C)\).
Finally, let \[N_\nu = L(P(\beta_\nu(\Ord)))\]
Granting that sufficiently complete ultrafilters on ordinals are idealized ordinals,
the models \(N_\nu\) are the corresponding idealizations of the inner model \(L(P(\Ord))\).
\begin{thm}\label{thm:inner_model}If there is a proper class of weakly Reinhardt cardinals, then for all sufficiently
    large cardinals \(\nu\), \(N_\nu\) contains a proper class of Reinhardt cardinals.\end{thm}
\begin{proof}
    Let $\nu$ be a cardinal large enough that for all ordinals $\alpha$, 
    $\beta_\nu(\alpha)$ can be wellordered. Let $N = N_\nu$. 
    We claim that if $\kappa > \nu$ is weakly Reinhardt, 
    then $\kappa$ is Reinhardt in $N$. To see this, let $j : V\to M$ 
    witness that $\kappa$ is weakly Reinhardt. 
    We will show that $j(N) = N$ and $j\restriction X\in N$ 
    for all $X\in N$. Hence $j\restriction N$ is an amenable 
    class of $N$ and in $N$, $j\restriction N$ is an elementary 
    embedding from the universe to itself. Letting $\mathcal C$ 
    denote the collection of classes amenable to $N$, it follows 
    that $(N,\mathcal C)$ is a model of NBG with a proper class of Reinhardt cardinals.

    We first show that $j(N) = N$, or in other words, 
    that $N$ is correctly computed by $M$. (Here we use that \(j(\nu) = \nu\)
    since \(\nu < \kappa\).) The closure properties of 
    $M$ guarantee that all ultrafilters on ordinals are in $M$, 
    and the elementarity of $j$ implies that for all $\alpha$, 
    $\beta_\nu(\alpha)$ is wellorderable in $M$. Finally, since $M$ 
    is closed under wellordered sequences, $P(\beta_\nu(\alpha))$ 
    is contained in $M$. This implies that $N$ is correctly computed by $M$.

    Finally, we show that for any $X\in N$, $j\restriction X\in N$. 
    For this, it suffices to show that for any ordinal $\alpha$, 
    $j\restriction P(\beta_\nu(\alpha))$ is in $N$. 
    Since $\beta_\nu(\alpha)$ 
    is wellorderable in $N$, it suffices to show that $j\restriction P(\delta)$ 
    belongs to $N$ where $\delta = |\beta_\nu(\alpha)|^N$. 
    Then letting \(f : P(\beta_\nu(\alpha))\to P(\delta)\) be a bijection in \(N\),
    \[j\restriction P(\beta_\nu(\alpha)) = j(f)^{-1}\circ (j\restriction P(\delta)) \circ f\] 
    and \(j(f)\in N\) since \(N = j(N)\) by the previous paragraph.
    But $j\restriction P(\delta)\in N$ because it is encoded by the extender 
    $E = \langle D_\gamma : \gamma < j(\delta)\rangle$ where 
    $D_\gamma$ is the ultrafilter on $\delta$ derived from $j$ 
    using $\gamma$:
    indeed, if $A\subseteq \delta$, then $j(A) = \{\gamma < j(\delta) : A\in D_\gamma\}$. 
    Since $E$ is a wellordered sequence of $\nu$-complete ultrafilters, 
    $E\in N$.
\end{proof}

We now show that ultrafilter Reinhardt cardinals are weakly Reinhardt,
so the same consistency strength lower bound applies to them.
\begin{prp}\label{thm:usc_reinhardt}
    If \(\kappa\) is ultrafilter Reinhardt, then \(\kappa\) is weakly Reinhardt.
    \begin{proof}
        Suppose \(j : V\to M\)
        is elementary and for all ordinals \(\alpha\),
        \(M^{\alpha}\subseteq M\) 
        and \(\beta(\alpha)\subseteq M\).
        We claim that for all ordinals \(\delta\), \(j\restriction P(\delta)\in M\).
        Consider the extender \(E = \langle D_\gamma : \gamma < j(\delta)\rangle\) 
        given by letting \(D_\gamma = \{A\subseteq \delta : \gamma \in j(A)\}\)
        be the ultrafilter derived from \(j\) using \(\gamma\). Then
        \(E\in M\), and hence \(j\restriction P(\delta)\in M\),
        since for \(A\in P(\delta)\), \(j(A) = \{\gamma < j(\delta) : A\in D_\gamma\}\).
    \end{proof}
\end{prp}

Finally, observe that the inner models considered here yield models with Reinhardt
cardinals that are a bit tamer than one might expect:
\begin{prp}
    If there is a proper class of Reinhardt cardinals, then
    there is an inner model with a proper class of Reinhardt cardinals
    in which every set is constructible
    from a wellordered sequence of ultrafilters on ordinals.\qed
\end{prp}

\begin{cor}
    If the existence of a proper class of Reinhardt cardinals is consistent, then
    it is consistent with \(V = L(P(P(\Ord)))\).
    \begin{proof}
        For any cardinal \(\lambda\), 
        a \(\lambda\)-sequence \(\langle S_\alpha : \alpha < \lambda\rangle\)
        of subsets of \(P(\lambda)\) can be coded by a single
        subset of \(P(\lambda\times \lambda)\); namely,
        \(\{\{\alpha\}\times A : A\in S_\alpha\}\).
        So if every set is constructible
        from a wellordered sequence of ultrafilters on ordinals,
        then \(V = L(P(P(\Ord)))\).
    \end{proof}
\end{cor}

More advanced techniques yield the following theorem, whose proof is omitted:
\begin{thm}
    If there is a Reinhardt cardinal,
    then for a closed unbounded class of cardinals \(\nu\), 
    there is a Reinhardt cardinal in \(N_\nu(V_{\nu+1})\).\qed
\end{thm}
Here \(N_\nu(V_{\nu+1})\) is the smallest inner model \(N\) such that
\(P(\beta_\nu(\Ord))\cup V_{\nu+1}\subseteq N\). We also note that
the proofs here easily generalize to show that
if there is a proper class of
Berkeley cardinals, then for sufficiently large \(\nu\),
there is a proper class of Berkeley cardinals in \(N_\nu\).
\section{Questions}
Let us list some variants of Schweber's original questions that
seem natural given the results of this note.
\begin{qst}
    Are Reinhardt cardinals and weakly Reinhardt cardinals equiconsistent?
\end{qst}
\begin{qst}
    Is the existence of a Reinhardt cardinal compatible with
    \(V = L(P(\Ord))\)?
\end{qst}
In the context of NBG, a cardinal \(\kappa\) is \textit{\(\Ord\)-supercompact}
if for all ordinals \(\alpha\), there is an elementary embedding
\(j : V\to M\) such that \(j(\kappa) > \alpha\) and \(M^\alpha\subseteq M\).
\begin{qst}
    Is NBG plus the existence of a proper class of \(\Ord\)-supercompact cardinals
    equiconsistent with ZFC plus a proper class of supercompact cardinals?
\end{qst}
\bibliographystyle{plain}
\bibliography{Bibliography}
\end{document}